\documentclass[12pt,a4paper]{amsart}
\usepackage{hyperref,amsfonts,amsmath,graphicx,amssymb}
\usepackage[sort,numbers]{natbib}

\newtheorem{theorem}{Theorem}
\newtheorem{lemma}{Lemma}
\newtheorem{corollary}{Corollary}
\newcommand{\ceil}[1]{\ensuremath{\protect\lceil#1\rceil}}
\newcommand{\half}{\ensuremath{\protect\tfrac{1}{2}}}

\begin{document}

\title[Colouring the Square of the Cartesian Product of Trees]{Colouring the Square of the\\ Cartesian Product of Trees}

\author{David R.\ Wood}
\address{\newline Department of Mathematics and
    Statistics\newline The University of Melbourne}
\email{woodd@unimelb.edu.au} 
\thanks{Supported by a QEII Research  Fellowship from the Australian Research Council.}



\begin{abstract}
We prove upper and lower bounds on the chromatic number 
of the square of the cartesian product of trees. The bounds 
are equal if each tree has even maximum degree. 
\end{abstract}

\maketitle

\section{Introduction}

This paper studies colourings of the square of cartesian products of
trees. For simplicity we assume that a tree has at least one edge. 
  
For our purposes, a \emph{colouring} of a graph $G$ is a function
$c:V(G)\rightarrow\mathbb{Z}$ such that $c(v)\neq c(w)$ for every edge $vw$ of
$G$. The \emph{square} graph $G^2$ of $G$ has vertex set $V(G)$, where
two vertices are adjacent in $G^2$ whenever they are adjacent in $G$
or have a common neighbour in $G$. A colouring of $G^2$ corresponds to a colouring
of $G$, such that in addition, vertices with a common neighbour in $G$
are assigned distinct colours. 

Let $[a,b]:=\{a,a+1,\dots,b\}$. The \emph{cartesian product} of graphs
$G_1,\dots,G_d$ is the graph $G_1\square\cdots\square G_d$ with vertex
set $\{(v_1,\dots,v_d):v_i\in V(G_i)\}$, where vertices
$v=(v_1,\dots,v_d)$ and $w=(w_1,\dots,w_d)$ are adjacent whenever
$v_iw_i\in E(G_i)$ for some $i\in[1,d]$, and $v_j=w_j$ for all $j\neq
i$. In this case, $vw$ is in \emph{dimension} $i$. Let $\Delta(G)$ be
the maximum degree of $G$. 

\begin{theorem}
  \label{Main} Let $T_1,\dots,T_d$ be trees.  Let $G:=T_1\square
  T_2\square\cdots\square T_d$.  Then $$1+\sum_{i=1}^d \Delta(T_i)
  \leq \chi(G^2)\leq 1+ 2\sum_{i=1}^d
  \ceil{\half\Delta(T_i)}\enspace.$$
\end{theorem}

This upper bound improves upon a similar bound by  \citet{JMV-IPL06},
who proved $\chi(G^2)\leq 1+2\sum_{i=1}^d(\Delta(T_i)-1)$, 
assuming that each $\Delta(T_i)\geq2$. Theorem~\ref{Main} implies:

\begin{corollary}
  Let $T_1,\dots,T_d$ be trees, such that $\Delta(T_i)$ is even for
  all $i\in[1,d]$. Let $G:=T_1\square T_2\square\cdots\square
  T_d$. Then $$\chi(G^2)=1+\sum_{i=1}^d \Delta(T_i)\enspace.$$
\end{corollary}

This corollary generalises a result of \citet{FGR-IPL03}, who proved
it when each $T_i$ is a path, and thus $G$ is a $d$-dimensional grid. See
\citep{SW10,FRR-JGT04,PorWood-Comb09,JMV-IPL06,CY08} for more related results.

\section{Proofs}

For a colouring $c$ of a graph $G$, the \emph{span} of an edge $vw$ of
$G$ is $|c(v)-c(w)|$.  The following lemma is well known; see
\citep{PorWood-Comb09} for example.

\begin{lemma}
  \label{Wrap} Let $G$ be a graph. If $G^2$ has a colouring in which
  every edge of $G$ has span at most $s$, then $G^2$ is
  $(2s+1)$-colourable.
\end{lemma}

\begin{proof} It is well known that a graph is $(k+1)$-colourable
  if and only if it has a colouring in which every edge has span at
  most $k$; see \citep{PorWood-Comb09}. If $G^2$ has a colouring in
  which every edge of $G$ has span at most $s$, then every edge of
  $G^2$ has span at most $2s$ in the same colouring. Thus $G^2$ is
  $(2s+1)$-colourable.
\end{proof}

\begin{lemma}
  \label{ColourTree} For every tree $T$ and non-negative integer $s$,
  $T^2$ has a colouring such that every edge of $T$ has span in
  $[s+1,s+\ceil{\half\Delta(T)}]$.
\end{lemma}

\begin{proof}
  We proceed by induction on $|V(T)|$. If $|V(T)|=2$ the result is
  trivial. Now assume that $|V(T)|\geq3$. Let $v$ be a leaf vertex of
  $T$. Let $w$ be the neighbour of $v$. By induction, $(T-v)^2$ has a
  colouring $c$ such that every edge of $T-v$ has span in
  $[s+1,s+\ceil{\half\Delta(T)}]$. Let $$X:=\{x\in\mathbb{Z}:|x|\in[s+1,s+\ceil{\half\Delta(T)}]\}\enspace.$$
  Each neighbour of $w$ in $T-v$ is coloured $c(w)+x$ for some $x\in
  X$.  Since $|X|\geq\Delta(T)$ and $w$ has degree less than
  $\Delta(T)$ in $T-v$, for some $x\in X$, no neighbour of $w$ is
  coloured $c(w)+x$.  Set $c(v):=c(w)+x$.  Thus $|c(v)-c(w)|=|x|\in
  [s+1,s+\ceil{\half\Delta(T)}]$. No two neighbours of $w$ receive the
  same colour. Hence $c$ is the desired colouring of $T$.
\end{proof}

\begin{proof}[Proof of Theorem~\ref{Main}]
  The lower bound is well known \citep{JMV-IPL06}.  In particular, for
  $i\in[1,d]$, let $v_i$ be a vertex of maximum degree in $T_i$.  Then
  $(v_1,\dots,v_d)$ has degree $\sum_i \Delta(T_i)$ in $G$.  This
  vertex and its neighbours in $G$ receive distinct colours in any
  colouring of $G^2$.  Thus $\chi(G^2)\geq 1+\sum_i \Delta(T_i)$.

  Now we prove the upper bound. Let $s_1:=0$ and
  $s_i:=\sum_{j=1}^{i-1} \ceil{\half\Delta(T_j)}$. By
  Lemma~\ref{ColourTree}, $T_i^2$ has a colouring $c_i$ such that
  every edge of $T_i$ has span in
  $[s_i+1,s_i+\ceil{\half\Delta(T_i)}]$. Thus the spans of edges in
  distinct trees are distinct.

  Colour each vertex $v=(v_1,\dots,v_d)$ of $G$ by
  $c(v):=\sum_{i=1}^dc_i(v_i)$.

  Suppose on the contrary that $c(v)=c(w)$ for some edge $vw$ of
  $G$. Say $vw$ is in dimension $i$. Thus $v_j=w_j$ for all $j\neq
  i$. Hence $c_i(v_i)=c_i(w_i)$, and $c_i$ is not a colouring of
  $G$. This contradiction proves that $c$ is a colouring of $G$.

  Suppose on the contrary that $c(x)=c(y)$ for two vertices $x$ and
  $y$ with a common neighbour $v$ in $G$. Say $vx$ is in dimension
  $i$, and $vy$ is in dimension $j$. Thus $v_\ell=x_\ell$ for all
  $\ell\neq i$, and $v_\ell=y_\ell$ for all $\ell\neq j$.  Now
  $c_i(x_i)-c_i(v_i)=c(x)-c(v)=c(y)-c(v)=c_j(y_j)-c_j(v_j)$.  Thus the
  edges $x_iv_i$ and $y_jv_j$ have the same span.  Since the spans of
  edges in distinct trees are distinct, $i=j$.  Hence
  $c_i(x_i)=c_i(y_i)$. However, $v_i$ is a common neighbour of $x_i$
  and $y_i$ in $T_i$, implying $c_i$ is not a colouring of
  $T_i^2$. This contradiction proves that $c$ is a colouring of $G^2$.

  Each edge of $G$ has span at most
  $\sum_{i=1}^d\ceil{\half\Delta(T_i)}$.  The result follows from
  Lemma~\ref{Wrap}.
\end{proof}


\begin{thebibliography}{6}
\providecommand{\natexlab}[1]{#1}
\providecommand{\url}[1]{\texttt{#1}}
\providecommand{\urlprefix}{}
\expandafter\ifx\csname urlstyle\endcsname\relax
  \providecommand{\doi}[1]{doi:\discretionary{}{}{}#1}\else
  \providecommand{\doi}{doi:\discretionary{}{}{}\begingroup
  \urlstyle{rm}\Url}\fi

\bibitem[{Chiang and Yan(2008)}]{CY08}
\textsc{Shih-Hu Chiang and Jing-Ho Yan}.
\newblock On {$L(d,1)$}-labeling of {C}artesian product of a cycle and a path.
\newblock \emph{Discrete Appl. Math.}, 156(15):2867--2881, 2008.
\newblock \urlprefix\url{http://dx.doi.org/10.1016/j.dam.2007.11.019}.

\bibitem[{Fertin et~al.(2003)Fertin, Godard, and Raspaud}]{FGR-IPL03}
\textsc{Guillaume Fertin, Emmanuel Godard, and Andr\'e Raspaud}.
\newblock Acyclic and {$k$}-distance coloring of the grid.
\newblock \emph{Inform. Process. Lett.}, 87(1):51--58, 2003.
\newblock \urlprefix\url{http://dx.doi.org/10.1016/S0020-0190(03)00232-1}.

\bibitem[{Fertin et~al.(2004)Fertin, Raspaud, and Reed}]{FRR-JGT04}
\textsc{Guillaume Fertin, Andr{\'e} Raspaud, and Bruce Reed}.
\newblock Star coloring of graphs.
\newblock \emph{J. Graph Theory}, 47(3):163--182, 2004.
\newblock \urlprefix\url{http://dx.doi.org/10.1002/jgt.20029}.

\bibitem[{Jamison et~al.(2006)Jamison, Matthews, and Villalpando}]{JMV-IPL06}
\textsc{Robert~E. Jamison, Gretchen~L. Matthews, and John Villalpando}.
\newblock Acyclic colorings of products of trees.
\newblock \emph{Inform. Process. Lett.}, 99(1):7--12, 2006.
\newblock \urlprefix\url{http://dx.doi.org/10.1016/j.ipl.2005.11.023}.

\bibitem[{P{\'o}r and Wood(2009)}]{PorWood-Comb09}
\textsc{Attila P{\'o}r and David~R. Wood}.
\newblock Colourings of the {C}artesian product of graphs and multiplicative
  {S}idon sets.
\newblock \emph{Combinatorica}, 29(4):449--466, 2009.
\newblock \urlprefix\url{http://dx.doi.org/10.1007/s00493-009-2257-0}.

\bibitem[{Sopena and Wu(2010)}]{SW10}
\textsc{Eric Sopena and Jiaojiao Wu}.
\newblock Coloring the square of the cartesian product of two cycles.
\newblock \emph{Discrete Math.}, 310(17-18):2327--2333, 2010.
\newblock \urlprefix\url{http://arxiv.org/abs/0906.1126}.

\end{thebibliography}

\def\cprime{$'$} \def\soft#1{\leavevmode\setbox0=\hbox{h}\dimen7=\ht0\advance
  \dimen7 by-1ex\relax\if t#1\relax\rlap{\raise.6\dimen7
  \hbox{\kern.3ex\char'47}}#1\relax\else\if T#1\relax
  \rlap{\raise.5\dimen7\hbox{\kern1.3ex\char'47}}#1\relax \else\if
  d#1\relax\rlap{\raise.5\dimen7\hbox{\kern.9ex \char'47}}#1\relax\else\if
  D#1\relax\rlap{\raise.5\dimen7 \hbox{\kern1.4ex\char'47}}#1\relax\else\if
  l#1\relax \rlap{\raise.5\dimen7\hbox{\kern.4ex\char'47}}#1\relax \else\if
  L#1\relax\rlap{\raise.5\dimen7\hbox{\kern.7ex
  \char'47}}#1\relax\else\message{accent \string\soft \space #1 not
  defined!}#1\relax\fi\fi\fi\fi\fi\fi} \def\Dbar{\leavevmode\lower.6ex\hbox to
  0pt{\hskip-.23ex\accent"16\hss}D}

\end{document}